\documentclass[11pt]{article}
\usepackage[a4paper, margin=2.5cm]{geometry}
\usepackage{graphicx} 
\usepackage{amsmath,amsthm,amssymb,mathtools,enumerate,enumitem,todonotes}
\usepackage[style=numeric-comp,sorting=nyvt,backend=bibtex,doi=false,isbn=false, maxbibnames=7]{biblatex}
\usepackage[defaultcolor=red]{changes}
\usepackage[colorlinks=true,unicode]{hyperref}
\usepackage{orcidlink}

\definecolor{ForestGreen}{rgb}{0.15,0.416,0.18}
\definecolor{EgyptBlue}{rgb}{0.063,0.2,0.65}
\hypersetup{
	linkcolor=EgyptBlue,         
	citecolor=EgyptBlue,          
	urlcolor=ForestGreen          
}

\usepackage{authblk}
\usepackage[capitalize, noabbrev]{cleveref}
\usepackage{bm}
\newtheorem{theorem}{Theorem}
\newtheorem{lemma}[theorem]{Lemma}

\newtheorem{thma}{Theorem}[section]

\theoremstyle{definition}
\newtheorem{definition}[theorem]{Definition}
\newtheorem{remark}[theorem]{Remark}

\numberwithin{theorem}{section}
\numberwithin{equation}{section}

\makeatletter
\newcommand{\subjclass}[2][2020]{%
  \let\@oldtitle\@title%
  \gdef\@title{\@oldtitle\footnotetext{#1 \emph{MSC Classification:} #2}}%
}
\newcommand{\keywords}[1]{%
  \let\@@oldtitle\@title%
  \gdef\@title{\@@oldtitle\footnotetext{\emph{Keywords:} #1}}%
}
\makeatother

\newcommand{\Z}{\mathbb{Z}}
\newcommand{\N}{\mathbb{N}}
\newcommand{\R}{\mathbb{R}}
\newcommand{\C}{\mathbb{C}}
\newcommand{\cC}{\mathcal{C}}
\newcommand{\cR}{\mathcal{R}}
\newcommand{\K}{\mathbb{K}}
\newcommand{\T}{\mathbb{T}}
\newcommand{\A}{\mathcal{A}}
\newcommand{\M}{\mathcal{M}}
\newcommand{\F}{\mathcal{F}}
\renewcommand{\S}{\mathcal{S}}
\newcommand{\U}{\mathcal{U}}
\newcommand{\V}{\mathcal{V}}
\renewcommand{\O}{\mathcal{O}}
\newcommand{\0}{\mathbf0}
\newcommand{\X}{\mathcal{X}}
\newcommand{\cS}{\mathcal{S}}
\newcommand{\eps}{\varepsilon}

\newcommand{\intd}{\ensuremath{\,\mathrm{d}}}

\DeclareMathOperator{\dom}{dom}

\DeclareMathOperator{\signc}{sc}

\title{Morse decomposition for semi-dynamical systems\\
with an application to  systems of state-dependent\\
 delay differential equations}
\author[1,2,3]{\orcidlink{0000-0003-2217-4960}\,István Balázs}
\author[1,2,3]{\orcidlink{0000-0001-9693-1923}\,Ábel Garab}
\author[4]{\orcidlink{0000-0003-3056-7086}\,Teresa Rauscher}
\date{\today}



\affil[1]{Bolyai Institute, University of Szeged,\par Aradi vértanúk tere 1, Szeged, H–6720, Hungary}

\affil[2]{HUN-REN--SZTE Analysis and Applications Research Group, \par Bolyai Institute,
University of Szeged}

\affil[3]{National Laboratory for Health Security, University of Szeged, Szeged, Hungary}

\affil[4]{Department of Mathematics, University of Klagenfurt, Universit\"atsstra{\ss}e 65--67, A--9020 Klagenfurt, Austria}
\begin{document}

\maketitle
\begin{abstract}
Understanding the structure of the global attractor is crucial in the field of dynamical systems, where Morse decompositions provide a powerful tool by partitioning the attractor into finitely many invariant Morse sets and gradient-like connecting orbits. Building on Mallet-Paret’s pioneering use of discrete Lyapunov functions for constructing Morse decompositions in delay differential equations, similar approaches have been extended to various delay systems, also including state-dependent delays. In this paper, we develop a unified framework assuming the existence and some properties of a discrete Lyapunov function for a semi-dynamical system on an arbitrary metric space, and construct a Morse decomposition of the global attractor in this general setting. We  demonstrate that our findings generalize previous results; moreover, we apply our theorem to a cyclic system of differential equations with threshold-type state-dependent delay.
\end{abstract}

\section{Introduction}
Under quite natural and frequently met assumptions, semi-dynamical systems -- whether induced by ordinary, delay, or parabolic differential equations, or difference equations -- are dissipative. This yields that their forward dynamics eventually enters a bounded subset of the state space, hence a global attractor $\A$ exists that contains all bounded entire solutions. Consequently, $\A$ encompasses all dynamically relevant objects, such as equilibria, periodic solutions,
as well as homo- or heteroclinic orbits. 

Beyond the mere existence of a global attractor, one is more interested in its internal structure.  A Morse decomposition serves as an effective tool for this task for the following reasons:
it allows disassembling the global attractor into finitely many invariant compact subsets (the
Morse sets) and their connecting orbits; the recurrent dynamics in $\A$ occurs entirely within the Morse sets; outside of them, the dynamics on $\A$ is gradient-like. 

Morse decompositions were introduced by Conley \cite{conley:78} and have been investigated by many scholars since then -- see, e.g.\ \cite{colonius:kliemann:00,kloeden:rasmussen:11} and references therein. To the best of our knowledge, Mallet-Paret \cite{mallet-paret:88-morse} was the first to employ a discrete Lyapunov function (i.e.\ an integer-valued function on the state-space that is nonincreasing along solutions) to construct a Morse decomposition, which he developed for a class of delay differential equations (DDEs) with negative feedback and constant delay 1. This discrete Lyapunov function $V$ is based on counting the number of sign changes on segments of solutions (these segments being elements of the infinite-dimensional state space $C([-1,0],\R)$). The Morse decomposition was constructed by showing and utilizing several properties of $V$.  

Mallet-Paret's ideas have been successfully applied to obtain similar results based on analogously defined discrete Lyapunov functions for delay difference equations \cite{garab:poetzsche:19} and for (systems of) DDEs \cite{polner:02,garab:20}, and, very recently, for scalar state-dependent delay differential equations \mbox{(SDDEs)} \cite{bartha:garab:krisztin:25}. All the proofs consist of the following main steps: 
\begin{enumerate}[label=(\arabic*)]
    \item defining an appropriate discrete Lyapunov function $V$, and verifying several properties of it (which are rather similar to the ones seen in \cite{mallet-paret:88-morse});
    \item defining the Morse sets (that are, roughly speaking, the sets of bounded entire solutions living within level sets of $V$);
    \item proving the existence of the Morse decomposition in a series of lemmas, similarly as done in \cite{mallet-paret:88-morse}.
\end{enumerate}

The aim of this paper is twofold: On the one hand, we define a framework that is sufficiently flexible to unify  all the results mentioned above. More precisely, we study a semi-dynamical system on an arbitrary metric space and assume that part (1) above is already done, that is, a discrete Lyapunov function with suitable properties exists. Then we carry out parts (2) and (3) in this general setting (see \cref{sec:gen-results}), and demonstrate that the results reproduce or refine those in \cite{mallet-paret:88-morse,garab:20,garab:poetzsche:19,bartha:garab:krisztin:25, polner:02} (see \cref{subsec:covers-prev-results}). 

On the other hand, in \cref{subsec:sdde} we study a cyclic system of SDDEs with a threshold-type delay and show that, based on the abstract results in \cref{sec:gen-results} and our recent findings on discrete Lyapunov functions for DDEs with variable delay \cite{balazs:garab:25}, a Morse decomposition of the global attractor can be constructed. 

There are other fields, such as tridiagonal ordinary differential equations \cite{avraham:Sharon:zarai:margaliot:20} and their discretizations \cite{mallet-paret:sell:03}, scalar reaction--diffusion equations \cite{brunovsky:fiedler:86,fiedler:mallet-paret:89}, or nonlinear Cauchy--Riemann equations \cite{van-den-berg:munao:vandervorst:16}, where analogously defined discrete Lyapunov functions are available. 
In addition, general dynamical systems and skew-product semi-flows that admit a discrete Lyapunov function have been studied (see, for example, \cite{terevsvcak:1994} and \cite{yan:zhou:25}, respectively). We hope that the findings obtained in \cref{sec:gen-results} complement these abstract results and further facilitate the construction of new Morse decompositions in these and related settings.
\medskip

In the next section, we introduce the necessary notation and recall some results from topology.

\section{Preliminaries}\label{sec:prelim}

\subsubsection*{Basic notation and definitions}
Let $\C,\R,\Z,\N_0$, and $\N$ denote the set of complex, real, integer, nonnegative integer, and positive integer numbers, respectively,  $(\X,d)$ be a metric space,
$$\T\in\{\R,\Z\},\qquad\T_+=\T\cap[0,\infty),\qquad\T_-=\T\cap(-\infty,0],$$
and let $S\colon \T_+\times \X \to \X$ be a semi-dynamical system having a global attractor $\A$, defined as follows.

\begin{definition}[{\cite[p.\ 1]{hale:88}}]
    The global attractor $\A$ is the maximal compact invariant set containing the limit set
    $$\omega(U)=\bigcap_{t\in\T}\overline{\bigcup_{s\ge t}S(s,U)}$$
    of all bounded sets $U\subseteq \X$.
\end{definition}

\begin{definition}
    A function $\varphi\colon\T\to \X$ is an entire solution of $S$ if
    $$\varphi(s+t)=S(t,\varphi(s))\quad\text{for all }s\in\T\text{ and }t\in\T_+.$$
\end{definition}

\begin{remark}
    The global attractor $\A$ of $S$ is the union of the trajectories of the bounded entire solutions of $S$.
\end{remark}

For convenience, we will use the following equivalent definition for $\omega$-limit sets of singletons. Besides that, let us define $\alpha$-limit sets of entire solutions in a similar fashion.

\begin{definition}
    Let $\xi\in \X$, then
    \begin{align*}
        \omega(\xi)&=\{\eta \in \X \mid  \exists(t_n)_{n=1}^\infty: S(t_n,\xi)\to \eta,\ t_n\to\infty,\ n\to\infty\}.
        \intertext{If $\varphi$ is an entire solution, then}
        \alpha(\varphi)&=\{\eta \in \X \mid \exists(t_n)_{n=1}^\infty: \varphi(t_n)\to \eta,\ t_n\to-\infty,\ n\to\infty\}.
    \end{align*}
\end{definition}

We proceed to define the Morse decomposition, which is the main object of the paper.

\begin{definition}
	A Morse decomposition of the global attractor $\A$ of the semi-dynamical system $S$ is a finite ordered collection $\M_0 \prec \M_1 \prec \dots \prec \M_m$ of nonempty, compact, invariant, pairwise disjoint subsets of $\A$ such that for any $\xi \in \A$ and any bounded solution $\varphi\colon \T \to \X$ of the semi-dynamical system $S$ through $\xi$, there exist $N,K$ in $\left\lbrace 0, \ldots, m\right\rbrace$ with $N \leq K$ and
	\begin{enumerate}[label=(M$_\arabic*$), leftmargin=*]
		\item \label{m1} $\alpha(\varphi) \subseteq \M_K$, $\omega(\xi) \subseteq \M_N$,
		\item \label{m2} $N=K$ implies that $\varphi(t)  \in \M_N$ for all $t \in \T$.
	\end{enumerate} 
We will refer to these properties as Morse properties.
\end{definition}

Although the Morse sets are -- by definition -- nonempty, for brevity, we take the liberty to occasionally list some sets $\M_0,\dots,\M_m$ as a Morse decomposition with some (but not all) of them potentially empty.

\subsubsection*{Ascoli's theorem}

The following theorem is a form of Ascoli's theorem for metric spaces that we will use several times in the following proofs.

\begin{thma}[{\cite[Theorem 47.1]{munkres:14}} Ascoli's theorem] \label{ascoli}
    Let $\U$ be a space, and let $(\V,d_\V)$ be a metric space. Give $\cC(\U,\V)$ the topology of compact convergence; let $\F$ be a subset of $\cC(\U,\V)$.
    \begin{enumerate}[label=(\alph*)]
        \item If $\F$ is equicontinuous under $d_\V$ and the set
        $$\F_a=\{f(a)\ |\ f\in\F\}$$
        has compact closure for each $a\in\U$, then $\F$ is contained in a compact subspace of $\cC(\U,\V)$.
        \item The converse holds if $\U$ is locally compact Hausdorff.
    \end{enumerate}
\end{thma}

Since we will work in metric spaces, the conclusion of Ascoli's theorem (part (a)) has the consequence that every infinite subset of $\F$ has a limit point (cf.\ \cite[Theorem~28.2]{munkres:14}).

\section{Abstract results}\label{sec:gen-results}

Using the notations introduced above, we assume that the semi-dynamical system $S$ has global attractor $\A$, and the following hypotheses hold throughout the paper. 

\begin{enumerate}[label=(H$_\arabic*$), leftmargin=*]
\item \label{cond-o-exists} The set $\O\subseteq \A$ is compact, and invariant in the sense that $S(t,\O)=\O$ for all $t\in \T_+$.
\item \label{cond-x-tilde} The set $\widetilde{\X}\subseteq \X$ is positively invariant (i.e.\ $S(\widetilde{\X},t)\subseteq \widetilde{\X}$), moreover, $\A \subseteq \widetilde{\X}$.
\item \label{cond-finer} A metric $\widetilde{d}$ is defined on $\widetilde{\X}$ such that 
 for any $\xi\in \widetilde{\X}$ and  $(\xi_n)_{n=1}^\infty \subseteq \widetilde{\X}$,
$$\lim_{n\to\infty}\widetilde{d}(\xi_n,\xi)=0\quad\text{implies}\quad\lim_{n\to\infty}d(\xi_n,\xi)=0.$$
\item  \label{att_d} Let $(\xi_n)_{n =1}^\infty$ be a sequence in $\A$ and $\xi \in \A$ with $\lim_{n \to \infty}d(\xi_n,\xi)=0$. Then there exists a subsequence $(\xi_{n_k})_{k=1}^{\infty}$ such that $\lim_{k\to\infty}\widetilde{d}(\xi_{n_k},\xi)=0$.
\end{enumerate}

Let us define our most important tool, the discrete Lyapunov function $$V\colon \dom(V) \to \N_0 \cup \lbrace \infty\rbrace,$$
where $\X\setminus \O \subseteq \dom(V)\subseteq \X $, and suppose that it has the following properties. 

\begin{enumerate}[label=(V$_\arabic*$), leftmargin=*]
	\item \label{v_usc} For every sequence $(\xi_n)_{n =1}^{\infty}\subseteq \X \setminus\O$  and $\xi \in \X \setminus\O$ with $\lim_{n \to \infty}d(\xi, \xi_n)=0$ it holds that 
	\begin{equation*}
		V(\xi) \leq \liminf_{n \to \infty} V(\xi_n).
	\end{equation*} 
	\item \label{v_cont} There exists a set $\cR \subseteq \widetilde{\X} \setminus \O$ such that for every sequence $(\xi_n)_{n =1}^{\infty}\subseteq \widetilde{\X} \setminus\O$  and $\xi \in \cR$ with $\lim_{n \to \infty} \widetilde{d}(\xi, \xi_n) = 0$ it holds that
	\begin{equation*}
		V(\xi) = \lim_{n \to \infty} V(\xi_n).
	\end{equation*} 
	\item \label{v_dec} For all $t \in \T_+$ and $\xi\in \X\setminus \O$ it holds that $V(S(t,\xi)) \leq V(\xi)$.
	\item \label{v_r} There exists $T \in \T_+$ such that for any $t \in \T_+$ with $t \geq T$, if $V(S(t,\xi))=V(\xi)$, then $S(t,\xi) \in \cR$.
	\item \label{v_n*} If $\O$ is nonempty, then there  exist $N^\ast\in \N_0$ and an open neighborhood $U \subseteq \X$ of $\O$ such that for any entire solution $\varphi\colon \T \to \X\setminus \O$ of the semi-dynamical system $S$ the following statements hold:
	\begin{itemize}
		\item[(a)] If there exists $t_0$ such that $\varphi(t) \in \overline{U} \cap \A$ for all $t\le t_0$, then $V(\varphi(t)) \leq N^*$ for all $t \in \T$.
		\item[(b)] If there exists $t_1$ such that $\varphi(t) \in \overline{U} \cap \A$ for all $t\ge t_1$, then $V(\varphi(t)) \geq N^*$ for all $t \in \T$.
	\end{itemize} 
\end{enumerate}

Property \ref{v_dec} explains the name ``discrete Lyapunov function''.

We note that in the definition of $V$, the set $\N_0$ could be replaced by any countable subset of isolated real numbers that is bounded below. For simplicity of notation, we work with the definition given above.

Now, let us define the sets that will form a Morse decomposition.
	If $\O\neq \emptyset$ ($\O=\emptyset$),  for   $N \in \N_0 \setminus \left\lbrace N^* \right\rbrace $ (for $N \in \N_0$) define the set
	\begin{align*}
		\S_N\coloneq \Big\lbrace \xi \in \A\setminus \O &: \text{there exists a bounded entire solution }  \varphi\colon \T \to \X \text{ through } \xi   \\
		&\mathrel{\hphantom{:}} \text{such that } V(\varphi(t)) =N \text{ for all } t \in \T, \text{ and }(\alpha(\varphi) \cup \omega (\xi))\cap\O=\emptyset\Big\rbrace,
        \end{align*}
	 and for a fixed $N_0>N^*$ ($N_0\in \N_0$), let us also define the set
    \begin{align*}
	\S_{N_0}^+\coloneq\Big\lbrace \xi \in \A \setminus \O &: \text{there exists a bounded entire solution }  \varphi\colon \T \to \X \text{ through } \xi  \\
	&\mathrel{\hphantom{:}}\text{such that } V(\varphi(t)) \geq N_0 \text{ for all } t \in \T, \text{ and } (\alpha(\varphi) \cup \omega(\xi))\cap\O=\emptyset \Big\rbrace.
	\end{align*}
    Furthermore, if $\O\neq \emptyset$, then let
    \begin{align*}
    \S_{N^*}\coloneq \Big\lbrace \xi \in \A\setminus \O &: \text{there exists a bounded entire solution }  \varphi\colon \T \to \X\\ 
	&\mathrel{\hphantom{:}} \text{though $\xi$ such that } V(\varphi(t)) =N^* \text{ for all } t \in \T\Big\rbrace \cup \O.
    \end{align*}
Finally, for $n\le N_0$ we define the Morse sets as
    $$\M_n\coloneq\begin{cases}
        \S_n,&\text{if }n<N_0,\\
        \S_{N_0}^+&\text{if }n=N_0.\\
    \end{cases}$$

Our first main theorem is the following.

\begin{theorem}\label{th:morse}
	If the semi-dynamical system $S$ has a global attractor $\A$ and hypotheses \emph{\ref{cond-o-exists}--\ref{att_d}} and \emph{\ref{v_usc}--\ref{v_n*}} are fulfilled, then the sets $\M_0,\dots ,\M_{N_0}$  form a Morse decomposition of the global attractor $\A$ for any $N_0>N^\ast$ $(N_0\in \N_0)$ in case of $\O\neq \emptyset$ $(\O=\emptyset)$. 
\end{theorem}

Let us first make some remarks and clarifications before turning to the proof.

\begin{remark}
Some of the assumptions might look artificial at first. However, they are constructed in a way that \cref{th:morse} can cover all the corresponding results of  \cite{mallet-paret:88-morse,garab:20,polner:02,garab:poetzsche:19,bartha:garab:krisztin:25}, as we demonstrate in \cref{sec:examples}.
\end{remark}

\begin{remark}
Note that our conditions permit the case $\widetilde{\X}=\X$ or $\widetilde{d}=d$. This is indeed the case when adapting the result to the discrete-time problem in \cite{garab:poetzsche:19} -- see \cref{subsec:covers-prev-results}. 
\end{remark}

\begin{remark}
In all of the examples mentioned, the discrete Lyapunov function is based on counting the number of sign changes on suitable sets. In these examples, $\O$ is the singleton of the origin (in the appropriate space), and $N^\ast$ is related to the number of unstable eigenvalues of the linearization. These -- together with other properties of the distribution of eigenvalues -- make it possible to show condition \ref{v_n*}. It is somewhat difficult to imagine that this condition holds for a different type of discrete Lyapunov function when $\O\neq \emptyset$; however, $\O=\emptyset$ is admissible, in which case condition \ref{v_n*} imposes no restriction.
\end{remark}

We need several lemmas to prove \cref{th:morse}. The next three lemmas are crucial for establishing the Morse properties.

\begin{lemma} \label{l1} Suppose that $\varphi\colon  \T \to \X$ is a bounded entire solution through $\xi \in \A$.
	\begin{itemize}
		\item[(a)]  If $N\in \N_0$ and $\lim_{t \to \infty} V(\varphi(t))=N$, then $V(\eta)=N$ for any $\eta \in \omega(\xi) \setminus \O$.
            \item[(b)] If $N\in \N_0\cup \{\infty\}$ and  $\lim_{t \to \infty} V(\varphi(t))\ge N$, then $V(\eta)\ge N$ for any $\eta \in \omega(\xi) \setminus \O$.
		\item[(c)]  If $N\in \N_0$ and $\lim_{t \to -\infty} V(\varphi(t))=N$, then $V(\eta)=N$ for any $\eta \in \alpha(\varphi) \setminus \O$.
            \item[(d)] If $N\in \N_0\cup \{\infty\}$ and $\lim_{t \to -\infty} V(\varphi(t))\ge N$, then $V(\eta)\ge N$ for any $\eta \in \alpha(\varphi) \setminus \O$.
	\end{itemize}
\end{lemma}
\begin{proof}
	We only prove statements (a) and (b) (simultaneously), since the others can be shown analogously. 
    
	For an arbitrary $\eta \in \omega(\xi) \setminus \O$, by hypothesis \ref{att_d}, there exists a sequence $( t_n)_{n =1}^{\infty}$  in $\T_+$ with $t_n \to \infty$ and
    \begin{equation}
        \lim_{n\to\infty}\widetilde{d}(\varphi(t_n),\eta)=0.\label{lim_d_tilde}
    \end{equation}
	
	Then, by the invariance of $\omega(\xi) \setminus \O$ there exists a bounded entire solution $\psi\colon \T \to \X  $ 
	 with $\psi(t) \in \omega(\xi) \setminus \O$ for all $t \in \T$ and $\psi(0)=\eta$. Hence, for any $t\in\T$, there exists $t_n'\to\infty$ with $\lim_{n\to\infty}\varphi(t_n')=\psi(t)$.

     Note that this, together with  \ref{v_usc}, implies  $V(\psi(t))\le N$ for all $t\in\T$, assuming the conditions of statement (a) are satisfied.
     
     Furthermore, by \ref{v_dec}, there exists $K\coloneq\lim_{t\to\infty} V(\psi(t))$. We claim that $K\geq N$.

     Assume indirectly that $K<N$, and hence, $K<\infty$. Then, \ref{v_r} implies the existence of $T>0$ such that $V(\psi(T))=K$ and $\psi(T)\in\cR$. Note that
     $$\lim_{n\to\infty}\widetilde{d}(\varphi(t_n+T),\psi(T))=0$$
     follows from \eqref{lim_d_tilde} and the continuity of $S$. Applying  \ref{v_r} again along with our assumption, we obtain
     $$N\le\lim_{n\to\infty}V(\varphi(t_n+T))=V(\psi(T))=K,$$
     which is a contradiction. Thus, $V(\eta)\ge K\ge N$, implying statement (a).

     It follows from $K\geq N$ and the monotonicity of $V$ that $V(\eta)\geq N$, as stated in part~(b).
\end{proof}

\begin{lemma} \label{l2}
	Suppose that $\O\neq \emptyset$ and $\varphi\colon\T \to \X$ is a bounded entire solution through $\xi \in \A$.
	\begin{itemize}
		\item[(a)] If $\lim_{t \to \infty} V(\varphi(t)) \neq N^*$, then either $\omega(\xi)\subseteq\O$ or $\omega(\xi)\subseteq\A\setminus\O$.
		\item[(b)] If $\lim_{t \to -\infty} V(\varphi(t)) \neq N^*$, then either $\alpha(\varphi)\subseteq\O$ or $\alpha(\varphi)\subseteq\A\setminus\O$. 
	\end{itemize}
\end{lemma}

\begin{proof}
	Again, we only prove part (a). Let us assume that $\omega(\xi)\cap\O\ne\emptyset$, but $\omega(\xi) \not\subseteq \O$, furthermore, let $U \subseteq \X$ be as in property \ref{v_n*}. Then there exists an open neighborhood $U_1$ of $\O\subseteq\A$ such that $U_1 \subset U$ and $\omega(\xi) \nsubseteq \overline{U_1}$. This, together with $\omega(\xi)\cap\O\ne\emptyset$, implies that $\varphi(t)$ enters and leaves $\overline{U_1}$ infinitely many times as $t \to \infty$.
    
    Accordingly, there exist positive sequences $t_n$, and $\tau_n$ in $\T_+$ such that $t_n$ satisfies
    $$\varphi(t_n)\in U_1,\quad d(\varphi(t_n),\O)\to0,\quad t_n\to\infty,\quad\text{as }n\to\infty,$$
    and $\tau_n$ is defined by
    \begin{equation*}
        \tau_n \coloneq\sup\left\{t\in\T_+: \varphi(u)\in U_1,\ \forall u\in[t_n,t_n+t) \cap \T \right\}.
    \end{equation*}
   Note that since $S(t,\O)=\O$ for all $t\in\T_+$, and $S$ is continuous, $\tau_n$ is well-defined for sufficiently large $n$, and we have $\tau_n\to\infty$ as $n\to\infty$. Moreover,  $\varphi(t_n+\tau_n)\notin U_1$. Then for each $n$, we can choose $s_n\in\T_+$ such that $\varphi(t_n-s_n)\notin U_1$, $\varphi(u)\in U_1$ for all $u \in (t_n-s_n, t_n+\tau_n)\cap \T$, and $t_n+\tau_n\le t_{n+1}-s_{n+1}$.
 
    Set 
    \begin{equation*}
        \psi_n(t)=\varphi(t_n-s_n+t) \quad \text{ and } \quad \vartheta_n(t)=\varphi(t_n+\tau_n+t) \quad \text{ for } t \in \T.
    \end{equation*}
    Now let us consider the continuous-time case ($\T=\R$). We want to apply \cref{ascoli} with
    $$\U=\R,\quad\V=\X,\quad d_\V=d,\quad\text{and}\quad\F=\{\psi_n: n\in\N\}\subseteq\cC(\R,\X),$$
    
    Obviously, the set
    \begin{equation*}
        \left\lbrace \psi_n(t):  n \in \N \right\rbrace = \left\lbrace \varphi(t_n-s_n+t):  n \in \N \right\rbrace \subseteq \A
    \end{equation*}
    is precompact for each $t\in\R$ as it is a subset of the compact global attractor $\A$. 
    To apply \cref{ascoli}, equicontinuity is needed. Therefore, we have to show that for all $\eps > 0$, there exists a $\delta>0$  such that $d(\psi_n(t), \psi_n(\overline{t}))< \eps$ for all $t$, $n$ with $|t - \overline{t}|< \delta.$
	
	For any $\eps>0$, the set $[0, \eps] \times \A$ is compact. Since the semi-dynamical system $S$ is continuous, the map $(s,\eta) \mapsto S(s,\eta)$ is uniformly continuous on $[0, \eps] \times \A$. Consequently, for every $\eps >0$, there exists a $\delta \in [0, \eps]$ such that $d(S(t, \eta), S(\overline{t}, \overline{\eta})) <\eps$ for all $(t, \eta),\, (\overline{t}, \overline{\eta}) \in [0, \eps] \times \A$ satisfying $|t-\overline{t}| < \delta$ and $d(\eta, \overline{\eta}) < \delta$.
    
In particular, we obtain
\begin{align*}
	d(\psi_n(t), \psi_n(\overline{t})) &= d(\varphi(t_n-s_n+t), \varphi(t_n-s_n+\overline{t})) \\
	&= d(S(t_n-s_n+t, \xi), S(t_n-s_n+\overline{t}, \xi)) \\
	&= d(S(t, S(t_n-s_n,\xi)),S(\overline{t}, S(t_n-s_n,\xi))) < \eps
\end{align*}
for all $|t-\overline t|<\delta$. Therefore, $\psi_n$ is equicontinuous and similarly, this also holds for $\vartheta_n$. 

Now, \cref{ascoli} yields the existence of a sequence of indices $(n_k)_{k \in \N_0}$ and maps $\psi, \vartheta\in\cC(\R,\X)$ such that 
$$\psi_{n_k}(t) \to \psi(t)\quad \text{and}\quad \vartheta_{n_k}(t) \to \vartheta(t)\quad \text{for all } t \in \R, \text{ as } k \to \infty.$$
In the discrete-time case, the same follows directly from (sequential) compactness of $\A$.

Since $\varphi\colon \T \to \X$ is a bounded entire solution and the semi-dynamical system is continuous, for all $s \in \T$, $t \in \T_+$, we obtain
\begin{align*}
	S(t, \psi(s)) & = S(t, \lim_{k \to \infty} \psi_{n_k}(s)) \\
	& = S(t, \lim_{k \to \infty} \varphi(t_{n_k}-s_{n_k} + s)) \\
	& = \lim_{k \to \infty} S(t,  \varphi(t_{n_k}-s_{n_k} + s)) \\
	& = \lim_{k \to \infty} \varphi(t_{n_k}-s_{n_k} + s + t) \\
	& = \lim_{k \to \infty} \psi_{n_k}(s + t)= \psi(s+t),
\end{align*}
that is, $\psi\colon \T \to \X$ is a bounded entire solution, and by an analogous argument, so is $\vartheta\colon \T \to \X$.
	
	Furthermore, since for all $n \in \N$,  $\psi_n(t)\in U_1$ for all $t\in(0,s_n+\tau_n)$, and $\vartheta_n(t) \in U_1$ for all $t\in(-s_n-\tau_n,0)$, we also have
	\begin{align*}
		\psi(t) \in \overline{U_1} \subset \overline{U}\quad \text{for all } t \in \T_+,\quad 
		\vartheta(t) \in \overline{U_1} \subset \overline{U}\quad \text{for all } t \in \T_-, 
	\end{align*}
	and 
	\begin{align*}
         \psi(0) \in \A \setminus U_1, \qquad \vartheta(0) \in \A \setminus U_1.
	\end{align*}
	Thus $\psi$ and $\vartheta$ are nontrivial entire solutions.
	By the monotonicity of $V$ and \ref{v_n*}, it follows that
	\begin{equation*} \label{eq1}
		V(\vartheta(t)) \leq N^* \leq V(\psi(t)) \quad \text{ for all } t \in \T.
	\end{equation*}
	In particular,
	\begin{equation*}
		V(\vartheta(0)) \leq N^* \leq V(\psi(0)).
	\end{equation*}
	But $\psi(0)$, $\vartheta(0) \in \omega(\xi) \setminus \O$, and by \cref{l1}, we conclude
	\begin{equation*}
		V(\psi(0))=V(\vartheta(0)) = \lim_{t \to \infty} V(\varphi(t)).
	\end{equation*}
	This altogether yields
	\begin{equation*}
		\lim_{t \to \infty}V(\varphi(t))=N^*,
	\end{equation*} 
	which is a contradiction. 
\end{proof}
\pagebreak

\begin{lemma} \label{l:v_n}
	Suppose that $\O\neq \emptyset$, $N$ is a nonnegative integer, and $\varphi\colon \T \to \X$ is a bounded entire solution through $\xi \in \A\setminus \O $.
	\begin{itemize}
		\item[(a)] If $\lim_{t \to \infty} V(\varphi(t))=N\neq N^*$, then either $\omega(\xi) \subseteq \O$ or $\omega(\xi) \subseteq \S_N$.
            \item[(b)] If $\lim_{t \to \infty} V(\varphi(t))\ge N>N^*$, then either $\omega(\xi) \subseteq \O$ or $\omega(\xi) \subseteq \S_N^+$. 
		\item[(c)] If $\lim_{t \to -\infty} V(\varphi(t))=N\neq N^*$, then either $\alpha(\varphi)\subseteq\O$ or $\alpha(\varphi) \subseteq \S_N$.
            \item[(d)] If $\lim_{t \to -\infty} V(\varphi(t))\ge N>N^*$, then either $\alpha(\varphi)\subseteq\O$ or $\alpha(\varphi) \subseteq \S_N^+$.
	\end{itemize}
\end{lemma}

\begin{proof}
	In order to prove assertion (a) (and (b)), let us assume that $\omega(\xi) \not\subseteq \O$ and that $\lim_{t \to \infty}V(\varphi(t))=N\neq N^*$ ($\lim_{t \to \infty}V(\varphi(t))\ge N>N^*$). According to \cref{l2},\linebreak $\omega(\xi) \cap\O=\emptyset$. Let $\eta \in \omega(\xi)$ be arbitrary. Since $\omega(\xi)$ is invariant, there exists a bounded entire solution $\psi\colon \T \to \X$ through $\eta$ such that $\psi(t) \in \omega(\xi)$ for all $t \in \T$. By compactness of $\omega(\xi)$, it follows that $\alpha(\psi) \cup \omega(\eta) \subseteq \omega(\xi)$. Therefore $(\alpha(\psi) \cup \omega(\eta))\cap\O=\emptyset$. 
    
	On the other hand, since $\psi(t) \in \omega(\xi)$ and by \cref{l1}, we have $V(\psi(t))=N$ ($V(\psi(t))\ge N$) for all $t \in \T$, so $\eta \in \S_N$ ($\eta \in \S_N^+$).

	The proof of statements (c) and (d) is analogous. 
\end{proof}

The next lemma shows that the Morse sets different from $\M_{N^\ast}$ are bounded away from $\O$, which is key in proving the compactness of the Morse sets.

\begin{lemma} \label{l:sn_bounded_away}
	 If $\O\neq \emptyset$, then for every $N \in \N_0 \setminus \{N^*$\}, there exists an open neighborhood $\widetilde{U}$  of $\O$ in $\A$ such that $\S_N \cap \widetilde{U} = \emptyset$. If $N>N^*$, then $\S_N^+\cap\widetilde{U}=\emptyset$.
\end{lemma}

\begin{proof}
	Assume for contradiction that there exists an $N \in \N_0 \setminus \{N^*\}$ and a sequence $(\xi_n)_{n =1}^{\infty}$ in $\S_N \setminus \O$ (in $\S_N^+ \setminus \O$) such that $d(\xi_n,\O)\to 0$ as $n \to \infty$. For each $n \in \N$, let $\varphi_n\colon \T \to \X$ be a bounded entire solution through $\xi_n$, such that $(\alpha(\varphi_n) \cup \omega(\xi_n))\cap\O=\emptyset$ and $V(\varphi_n(t))=N$ for all $t \in \T$. 
    
	Case $N >N^*$. We claim that $\omega(\xi_n)$ cannot be a subset of $\overline{U}$, where $U \subseteq \X$ is as in property \ref{v_n*}.
    
	To see this, suppose instead that $\omega(\xi_n) \subseteq \overline{U}$. Let $\eta \in \omega(\xi_n)$. By invariance of $\omega(\xi_n)$, there exists a bounded entire solution $\psi_n\colon \T \to \X$ through $\xi_n$, such that $\psi_n(t) \in \omega(\xi_n) \subseteq \overline{U}$ for all $t \in \T$. Then \ref{v_n*} implies  $V(\psi_n(t)) = N^*$  for all $t \in \T$, but this contradicts $V(\eta)=N$, which should also hold by \cref{l:v_n}. Hence,  $\omega(\xi_n) \nsubseteq \overline{U}$ for all $n \geq 0$, as claimed. 
    
	Since $d(\xi_n,\O)\to 0$, we can define the sequence
    $$s_n=\sup\{s\in\T: \varphi_n(t)\in U,\ t\in[0,s) \cap \T\}.$$
	
	As the semi-dynamical system $S$ is continuous and satisfies $S(t,\O)=\O$ for all $t \in \T_+$, it follows that $s_n \to \infty$ as $n \to \infty$. Let $\psi_n(t)\coloneq \varphi_n(t+s_n)$ for all $t \in \T$.  By \cref{ascoli}, there exists a subsequence $\psi_{n_k}$ and an entire solution $\psi$ such that  $\psi_{n_k}(t) \to \psi(t)$ for all $t \in \T$ as $k \to \infty$. 
    
	The solution $\psi$ is nontrivial because $\psi(0)=\lim_{n \to \infty} \varphi_n(s_n) \not\in U$. On the other hand, $\psi(t)=\lim_{n \to \infty} \varphi_n(t+s_n) \in \overline{U}$ for all $t \leq 0$.
	
	Hence, property \ref{v_n*} yields $V(\psi(t)) \leq N^*$ for all $t \in \T$. 
    
	Finally, by \ref{v_dec} there exists some $T>0$, such that $\psi(T) \in \cR$. Since
	\begin{equation*}
		\varphi_{n_k}(s_{n_k}+T) = \psi_{n_k}(T) \to \psi(T) \quad \text{ for } k \to \infty,
	\end{equation*}
	we obtain by virtue of property \ref{v_cont} -- and by passing to a subsequence if necessary -- that
	\begin{equation*}
		N=\lim_{k \to \infty} V(\varphi_{n_k}(s_{n_k}+T)) = V(\psi(T)) \leq N^*,
	\end{equation*}
	 which contradicts $N>N^\ast$ and proves our statement. 
    
	The proof for the case $N < N^*$ is analogous. 
\end{proof}

\begin{lemma} \label{l6}
	The sets $\S_{N^*}$, $\S_N$, and $S_{N_0}^+$ are closed for all nonnegative integers $N$ and $N_0$ with $N_0>N^*$.
\end{lemma}

\begin{proof}
	Let $(\xi_n)_{n =1}^{\infty}$ be a sequence in $\S_N$ for some $N \in \N$, such that $\lim_{n \to \infty} \xi_n=\xi$ for some $\xi \in \A$. We claim that $\xi \in \S_N$. 
    
	If $\xi\in\O$, then \cref{l:sn_bounded_away} yields $N=N^*$. By definition of $\S_{N^*}$, $\O\subseteq\S_{N^*}$ also holds, so the statement is proved in this particular case. 
    
	Assume now that $\xi\not\in\O$. By the definition of $\S_N$ and using $\xi_n \in \S_N$, there exist bounded entire solutions $\varphi_n\colon \T \to \X$ through $\xi_n$, such that $\varphi_n(t) \in \S_N$ and $V(\varphi_n(t))=N$ for all $n \in \N$ and $t \in \T$. The compactness of $\A$ and \cref{ascoli} yield a subsequence $\varphi_{n_k}$ and a solution $\varphi\colon \T \to \X$ such that $\varphi_{n_k}(t) \to \varphi(t)$ for every $t \in \T$ as $k \to \infty$. 
    
	Necessarily, $\varphi(0)=\xi$, and by the invariance of $\A\setminus\O$, this yields  $\varphi(t) \in \A \setminus \O$ for all $t \in \T$. Applying \ref{v_r}, we obtain $s_1, s_2 \in \T$ with $s_1 < 0 < s_2$ and $\varphi(s_1),\varphi(s_2)\in\cR$, such that
	\begin{equation*} 
		V(\varphi(s_1)) =\lim_{t \to -\infty} V(\varphi(t)), \qquad V(\varphi(s_2))= \lim_{t \to \infty} V(\varphi(t)).
	\end{equation*}
	 On the other hand, using \ref{att_d}, we may assume -- after passing to a subsequence if necessary~-- that
    $$\widetilde{d}(\varphi_{n_k}(s_1), \varphi(s_1)) \to 0\quad\text{and}\quad\widetilde{d}(\varphi_{n_k}(s_2),\varphi(s_2)) \to 0\quad\text{as }k \to \infty,$$
    so \ref{v_cont} implies $V(\varphi(s_1))=V(\varphi(s_2))=N$. Therefore, we obtain $V(\varphi(t))=N$ for all $t \in \T$. 
    
	If $N=N^*$ or $\O=\emptyset$, then $\xi \in \S_N$ is already established. 
    
	Otherwise, it remains to prove that $(\alpha(\varphi) \cup \omega(\xi))\cap\O=\emptyset$.  In this case, \cref{l:sn_bounded_away} gives an open neighborhood $\widetilde{U}$ of $\O$ in $\A$ such that $\varphi_n(t) \in \A \setminus \widetilde{U}$  for all $t \in \T$ and $n \in \N$. Hence,
	\begin{equation*}
		\varphi(t) \in \overline{\A \setminus \widetilde{U}} = \A \setminus \widetilde{U}
	\end{equation*}
	 for all $t \in \T$, which leads to $\alpha(\varphi) \cup \omega(\xi) \subseteq \A\setminus \widetilde{U}$. This shows that $\alpha(\varphi) \cup \omega(\xi)\cap\O=\emptyset$, completing the proof of the claim.

    Let us now assume  that $N_0>N^\ast$, $\xi_n\in\S_{N_0}^+$ with $\lim_{n\to\infty}\xi_n=\xi\in\A$, and $\varphi_n$ is a solution through $\xi_n$ for all $n\in\N$. Then $V(\xi_n)\ge N_0$ for all $n\in\N$.

    As shown earlier, using \cref{ascoli}, we may suppose that there exists a bounded entire solution $\varphi$ with $\varphi(0)=\xi$ such that $\varphi_n\to\varphi$. We have to show that $V(\varphi(t))\ge N_0$ for all $t\in\T$.

    Assume indirectly that there exists $t_0\in\T$ with $V(\varphi(t_0))<N_0$. Then \ref{v_r} implies the existence of $T>0$ such that
    $$\lim_{t\to\infty}V(\varphi(t))=V(\varphi(T))\quad\text{and}\quad\varphi(T)\in\cR.$$
    By \ref{att_d}, we may suppose that $\lim_{n\to\infty}\widetilde{d}(\varphi_n(t),\varphi(t))=0$ and hence, we conclude
    $$N_0>V(\varphi(T))=V(\varphi_n(T))\ge N_0,$$
    contradicting the statement.
\end{proof}

Now, we are ready to prove our main theorem.

\begin{proof}[{\bf Proof of \cref{th:morse}}]
	By definition, the sets $\M_n$ are pairwise disjoint and invariant. Since $\A$ is compact, \cref{l6} yields that $\M_n$ is compact for all possible $n$. It remains to prove that these sets satisfy the Morse properties \ref{m1}, \ref{m2}, that is, for every $\xi \in \A$ and any bounded entire solution $\varphi\colon \T \to \X$ through $\xi$, there exist $N$ and $K$ with $N \leq K$ such that $\alpha(\varphi) \subseteq \M_K$, $\omega(\xi) \subseteq \M_N$, and in case $N=K$, $\xi \in \M_N$, thus $\varphi(t)  \in \M_N$ for all $t \in \T$. If $\O=\emptyset$, this follows directly from \cref{l1}, so we may assume $\O\neq \emptyset$ for the remainder of the proof.
    
    If $\xi\in \O$, then invariance and compactness of $\O$, together with the definition of $\M_{N^\ast}$, imply that both Morse properties hold.
    In order to show this for $\xi \in \A \setminus \O$ as well, let $\varphi\colon \T \to \X$ be a bounded entire solution through $\xi$. Furthermore, define
	\begin{equation*}
		K:=\lim_{t \to -\infty} V(\varphi(t)) \quad \text{ and } \quad N:=\lim_{t \to \infty} V(\varphi(t)).
	\end{equation*}
	Note that from the monotonicity of $V$, we obtain  $N \leq K$.
    
	First, observe that if $N=N^*$, then $\omega(\xi) \subseteq \M_{N^*}$. In order to prove this, choose $\eta \in \omega(\xi)$ arbitrarily. If $\eta\in\O$, then $\eta \in \M_{N^*}$ by definition, thus we may assume $\eta \notin \O$. By \cref{l1}, $V(\eta)=N^*$. Moreover, by the invariance of $\omega(\xi) \setminus \O$, there exists an entire solution $\psi\colon \T \to \omega(\xi) \setminus \O $ through $\eta$, for which \cref{l1} yields  $V(\psi(t))=N^*$ for all $t \in \T$, so $\eta \in \M_{N^*}$.

	A similar argument shows that $K=N^*$ implies  $\alpha(\varphi) \subseteq \M_{N^*}$.
    
	Next, we distinguish four cases based on the values of $N$ and $K$.
    \medskip
    
	\noindent$\textbf{Case  1.}$ If $N=K=N^*$, then $\alpha(\varphi) \cup \omega(\xi) \subseteq \M_{N^*}$ by our previous observation. Moreover, from the monotonicity of $V$, it follows that $V(\varphi(t))=N^*$ on $\T$. This implies $\xi \in \M_{N^*}$, thus both Morse properties hold.
    \medskip
    
	\noindent$\textbf{Case 2.}$ Assume $N=N^*<K$. As  shown above, $\omega(\xi) \subseteq \M_{N^*}$. Moreover, observe that $\alpha(\varphi) \not\subseteq \O$; otherwise property \ref{v_n*} would imply $V(\varphi(t)) \leq N^*$ for all $t \in \T$, and thus $K \leq N^* =N$, which contradicts $K>N$. Therefore, \cref{l:v_n} can be applied to obtain $\alpha(\xi) \subseteq \M_K$, so property \ref{m1} is fulfilled. Property \ref{m2} follows automatically since the two Morse sets in question, i.e. $\M_{N^*}$ and $\M_K$, are distinct.
    \medskip
    
	\noindent$\textbf{Case 3.}$ A similar argument applies in the case when $N <N^* =K$.
   \medskip
    
	\noindent$\textbf{Case 4.}$ In the case $N \neq N^* \neq K$, let
    $$K_1\coloneq \begin{cases}
        K,&\text{if }K<N_0,\\
        N_0,&\text{otherwise,}
    \end{cases}\qquad N_1\coloneq \begin{cases}
        N,&\text{if }N<N_0,\\
        N_0,&\text{otherwise.}
    \end{cases}$$
    Then \cref{l:v_n} yields that either $\omega(\xi)\subseteq\O$ or $\omega(\xi) \subseteq \M_{N_1}$; similarly, either $\alpha(\varphi)\subseteq\O$ or $\alpha(\varphi) \subseteq \M_{K_1}$. Note that $\omega(\xi)$ and $\alpha(\varphi)$ cannot both be in $\O$, since property \ref{v_n*} would then imply $V(\varphi(t)) \equiv N^*$ on $\T$, contradicting  $N_1 \neq N^* \neq K_1$.
    
	If $\omega(\xi) \not\subseteq \O \not\supseteq \alpha(\varphi)$, then \cref{l:v_n} yields  $\omega(\xi) \subseteq \M_{N_1}$ and $\alpha(\varphi) \subseteq \M_{K_1}$, so \ref{m1} is fulfilled. If $N_1=K_1$, then $V(\varphi(t))=N =K$ ($V(\varphi(t))\geq N_0$) for all $t \in \T$ in case of $K_1=N_1<N_0$ ($K_1=N_1 = N_0$). Now, our assumption on the limit sets with \cref{l2} ensure that $(\alpha(\varphi) \cup \omega(\xi))\cap\O=\emptyset$, thus $\varphi(t) \in \M_{N_1}$ also holds for all $t \in \T$. This establishes property \ref{m2}.
    
	If $\omega(\xi)\subseteq\O \not\supseteq \alpha(\varphi)$, then $\omega(\xi) \subseteq \M_{N^*}$ by definition. Furthermore, property \ref{v_n*} implies $V(\varphi(t)) \geq N^*$ for all $t \in \T$, and consequently $N^* < N_1 \leq K_1$. On the other hand,  \cref{l:v_n} yields  $\alpha(\varphi) \subseteq \M_{K_1}$, so \ref{m1}  holds again, and the \ref{m2} is automatically fulfilled.
    
	An analogous argument applies to the case  $\omega(\xi) \not\subseteq\O\supseteq \alpha(\varphi)$.
    
	We have now taken all possible cases into consideration, so the proof is complete. 
\end{proof}

\section{Applications}\label{sec:examples}

In the first part of this section, we present the second main result of the paper: in \cref{subsec:sdde}, we study a cyclic system of differential equations with a certain class of state-dependent delays, to which \cref{th:morse} can be applied to obtain a Morse decomposition of the global attractor, thanks to recent findings from our paper \cite{balazs:garab:25} on a discrete Lyapunov function for DDEs with time-variable delays.

\subsection{A cyclic system of differential equations with state-dependent delay}\label{subsec:sdde}

Let $N$ be a nonnegative integer, and consider the system
\begin{equation}\label{eq:sdx}
\begin{aligned}
    \dot x^i(t)&=f^i(x^i(t),x^{i+1}(t)),&\quad& i\in\{0,\dots, N-1\},\\
    \dot x^N(t)&=f^N(x^N(t),x^0(t-\tau(x_t)))&&
\end{aligned}
\end{equation}
under the following hypotheses:
\begin{enumerate}[label=(S$_\arabic*$), leftmargin=*]
	\item \label{hyp_delta_aut} Each $f^i$ is $C^1$-smooth on $\R^2$, and fulfills the feedback assumptions
		\begin{align*}
				vf^i(0,v)&>0,
             & D_2 f^i(0,0)&>0,&i\in \{0,\dots, N-1\},\\
			\delta vf^N(0,v)&>0,
               &\delta D_2 f^N(0,0)&>0,&
		\end{align*}
		 for all $v\neq 0$, where $\delta\in\{\pm1\}$ and $D_2$ stands for differentiation w.r.t.\ the second argument.
    \item \label{hyp_dissip}  There exists $M>0$ such that 
\begin{alignat}{3}
	f^i(u,v)&<0 &&\quad  \text{if } u\geq M \text{ and } |v|\leq u\label{aut_dissip_1}\\
\shortintertext{and}
	f^i(u,v)&>0 &&\quad \text{if } u\leq -M \text{ and } |v|\leq |u|\label{aut_dissip_2}
\end{alignat}
for all $0\leq i\leq N$.
\end{enumerate}

Consider the Banach space $C(\K,\R)$ equipped with the maximum norm:
$$\|\varphi\|=\max_{s\in\K}|\varphi(s)|,\quad\varphi\in  C(\K,\R),$$
where $\K\coloneq [-r ,0]\cup\{0,\ldots,N\}$ for some positive constant $r $, and let $x_t\in C(\K,\R)$ be defined by
$$x_t(s)=\begin{cases}
    x^0(t+s),&\text{if } s \in [-r , 0],\\
    x^s(t),& \text{if } s\in \{1,\dots,N\}.
\end{cases}$$

Let 
\[L_0\coloneq \max \{|f^i(u,v)| : (u, v) \in [-M, M]^2,\ i\in \{0,\dots,N\}\} ,\]
and set the phase space 
\begin{equation*}
	\X \coloneq \{\varphi \in C(\K,\R) \colon \|\varphi\| < M \text{ and $\varphi$ is $L_0$-Lipschitz}\}.
\end{equation*}

Let $d$ be the metric on $\X$ induced by the maximum norm. Moreover, let 
\[\widetilde{\X}\coloneq \{\varphi \in \X : \mbox{the restriction } \varphi|_{[-r ,0]} \mbox{ is continuously differentiable}\}\]
with the norm 
\[\|\varphi\|_{1}\coloneq \|\varphi\|+ \sup_{s\in[-r ,0]}|\dot{\varphi}(s)|,
\]
and let $\widetilde{d}$ be the metric induced by this norm.

Assume further that $\tau\colon  \X \to \R$ is defined by
        \begin{equation}
            \int_{-\tau(\varphi)}^0 \alpha(\varphi(s))\intd s=1,\label{eq:thr_delay}
        \end{equation}
        where $\alpha\colon\R\to [\alpha_1,\alpha_2]$ is continuous and $\alpha_1\coloneq 1/r  \leq \alpha_2$. Then $\tau(\varphi)$ is well-defined  and bounded by 
        $$ \frac{1}{\alpha_{2}}  \leq \tau (\varphi) \leq r $$
for all $\varphi\in \X$ (cf.\ \cite[Proposition\ 6.1]{bartha:garab:krisztin:25}). 

For the sake of simplicity, we also make the following technical assumption.
\begin{enumerate}[label=(D$_{\arabic*}$)]
    \item  There exist $\alpha_0$ and $\eps>0$ such that $\alpha(x)=\alpha_0$ for all $|x|<\eps$. \label{delay_const_near_0}
\end{enumerate}

It is known that $\varphi\mapsto \tau(\varphi)$ is globally Lipschitz \cite[Section~4.1]{balazs:garab:25}. One can easily prove, analogously to the scalar case (cf.\ Propositions 2.1 and 2.4 in \cite{bartha:garab:krisztin:25}), that for any $\varphi\in \X$, there exists a unique forward solution $x^{\varphi}\colon [0,\infty)\to \R^{N+1}$ such that $x^{\varphi}_0=\varphi$, and $(t,\varphi)\mapsto S(t,\varphi)\coloneq x^{\varphi}_t$ defines a continuous semi-dynamical system that has a connected global attractor $\A \subset \widetilde{\X}$. Moreover, for any fixed solution $x\colon [t_0,\infty)\to \R^{N+1}$, the mapping $t\mapsto \eta(t)\coloneq t-\tau(x_t)$ is strictly increasing. 

	Let $\O\coloneq \{\bf 0\}$, where ${\bf 0} \in \X$ denotes the constant zero function. For $\varphi\in \X \setminus \O$ and $a\in [-r ,0)$, introduce the notation
	\begin{align*}
		V^+(\varphi,a)&=\begin{cases}
			\signc(\varphi,a),&\text{if }\signc(\varphi,a)\text{ is even or infinite,}\\
			\signc(\varphi,a)+1,&\text{if }\signc(\varphi,a)\text{ is odd,}
		\end{cases}\\
		V^-(\varphi,a)&=\begin{cases}
			\signc(\varphi,a),&\text{if }\signc(\varphi,a)\text{ is odd or infinite,}\\
			\signc(\varphi,a)+1,&\text{if }\signc(\varphi,a)\text{ is even,}
		\end{cases}
	\end{align*}
	where $\signc(\varphi,a)$ denotes the number of sign changes of  $\varphi$ on the set $[a,0]\cup\{1,\dots,N\}$, that is,
	\begin{align}
		\begin{aligned}
			\signc(\varphi,a)=\sup\{k\in \N:{}&\text{there exist }\theta_i\in [a,0]\cup\{1,\dots,N\},\ 0\le i\le k,\\
			&\text{with }\theta_i<\theta_{i-1}\text{ and }\varphi(\theta_i)\varphi(\theta_{i-1})<0\text{ for }1\le i\le k\},
		\end{aligned}\label{def:sc}
	\end{align}
	assuming here $\sup\emptyset=0$. Let $\delta\in\{\pm1\}$ be fixed so that $f$ satisfies \ref{hyp_delta_aut}. For $N\geq 1$ and $a\in [-r ,0)$,  introduce the sets
    \begin{align*}
		\cS^0&\coloneq\Bigl\{\varphi\in \widetilde{\X}:\text{if }\varphi(0)=0\text{, then }\dot\varphi(0)\varphi(1)>0\Bigr\},\\
		\cS_a&\coloneq\Bigl\{\varphi\in \widetilde{\X}:\text{if }\varphi(a)=0\text{, then }\delta\varphi(N)\dot\varphi(a)<0\Bigr\},\\
		\cS^*_a&\coloneq\Bigl\{\varphi\in \widetilde{\X}:\text{if }\varphi(s)=0\text{ for some }s\in[a,0]\text{, then }\dot\varphi(s)\ne0\Bigr\},\\
		\cS^N_a&\coloneq\Bigl\{\varphi\in \widetilde{\X}:\text{if }\varphi(N)=0\text{, then }\delta\varphi(N-1)\varphi(a)<0\Bigr\},\\
        \intertext{furthermore, for $N\ge 2$ and $1\le i\le N-1 $, let} 
    	\cS^i&\coloneq\Bigl\{\varphi\in \widetilde{\X}:\text{if }\varphi(i)=0\text{, then }\varphi(i-1)\varphi(i+1)<0\Bigr\},\\
	\shortintertext{and finally, let} 
	\cR_a&\coloneq \cS_a\cap\Biggl(\bigcap_{i=0}^{N-1} \cS^i\Biggr)\cap \cS^N_a\cap \cS^*_a \subseteq \widetilde{\X}\setminus \O.
    \end{align*}

    For clarity, in the case $N=0$ (scalar case), we only define the sets $\cS_a$, $\cS^*_a$, and $\cS^0$, where in the latter set we replace $\varphi(1)$ by $\varphi(a)$, and set $\cR_a=\cS^0\cap \cS_a \cap \cS^*_a$.
    
    Let us introduce the notation 
    \[\cR \coloneq \Bigl\{\varphi \in \widetilde{\X} : \varphi \in \cR_{-\tau(\varphi)}\Bigr\},\]
    and for an arbitrary $\varphi\in \X \setminus \O$ let 
    \begin{align}
        V(\varphi)\coloneq\begin{cases}
		V^+(\varphi,-\tau(\varphi)),&\text{if }\delta=1,\\
		V^-(\varphi,-\tau(\varphi)),&\text{if }\delta=-1.
	\end{cases}\label{v_pm}
    \end{align}

Furthermore, we introduce the linearization of \eqref{eq:sdx} at $\0$, namely
\begin{equation}\label{linearized_eq}
\begin{aligned}
\dot{x}^i(t)&=\mu^i x^i(t) + \gamma^i x^{i+1}(t),&\quad 0\leq i\leq N-1\\
\dot{x}^N(t)&=\mu^N x^N(t) + \gamma^N x^0(t-\tau(\0)),&
\end{aligned}
\end{equation}
where $\mu^i= D_1 f^i(0,0)$, and $\gamma^i=D_2 f^i(0,0)$. Then $\delta\gamma^N>0$ and $\gamma^i\in (0,\infty)$ for all $0\leq i\leq N-1$. The eigenvalues of \eqref{linearized_eq} are precisely the numbers $\lambda\in\C$ that solve the characteristic equation 
\begin{equation}\label{char_eq}
\prod_{i=0}^N (\lambda-\mu^i) - e^{-\lambda\tau(\0)}\prod_{i=0}^N \gamma^i=0.
\end{equation}

Let us denote by $M^\ast$  the number of complex roots of  \eqref{char_eq} (counting with multiplicities)  that have strictly positive real part. Note that, according to \cite[Lemma~7.4]{mallet-paret:sell:96-lyapunov}, $M^\ast$ is always a finite number. Moreover, let
\begin{align*}
N^\ast&\coloneq\begin{cases}
M^\ast+1,&\mbox{if the origin is nonhyperbolic and $M^\ast$ is odd},\\
M^\ast,&\mbox{otherwise},
\end{cases}\\
\intertext{in the case of positive feedback, i.e.\ $\delta=1$, and}
N^\ast &\coloneq\begin{cases}
M^\ast+1,&\mbox{if the origin is nonhyperbolic and $M^\ast$ is even},\\
M^\ast,&\mbox{otherwise},
\end{cases}
\end{align*}
in the case of negative feedback, i.e.\ $\delta=-1$.

\begin{theorem}\label{thm:sdde}
Using the above definitions and notations, conditions \emph{\ref{cond-o-exists}--\ref{att_d}} and \emph{\ref{v_usc}--\ref{v_n*}} are satisfied and a Morse decomposition of the global attractor of system \eqref{eq:sdx}--\eqref{eq:thr_delay} given in \cref{th:morse} exists.
\end{theorem}
\begin{proof} We check all the conditions one-by-one below.
\medskip

\noindent \textbf{\ref{cond-o-exists}} is shown in \cite[Lemma~2.2\,(iii)]{balazs:garab:25}, where $\eta(t)\coloneq t-\tau(x_t)$. 
\medskip

\noindent\textbf{\ref{cond-x-tilde}--\ref{cond-finer}}  hold trivially.
\medskip

\noindent\textbf{\ref{att_d}} The state-dependency has no role in this statement, so the argument presented in \cite[Lemma 4.3]{polner:02} applies with straightforward modifications (cf.\ also \cite[Lemma~3.14]{garab:20}).
\medskip

\noindent\textbf{\ref{v_usc}--\ref{v_cont}} follow directly from \cite[Lemma~2.1]{balazs:garab:25}.
\medskip

\noindent\textbf{\ref{v_dec}--\ref{v_r}} are direct applications of \cite[Theorems~3.3--3.4]{balazs:garab:25}, respectively.
\medskip

\noindent \textbf{\ref{v_n*}}  The proof of this condition builds on the relation of equation \eqref{eq:sdx} and its linearization around the trivial solution, following the argument of  \cite[Lemma~4.4]{bartha:garab:krisztin:25} (or \cite[Lemma~3.9]{polner:02}).

The proofs of both statements of \ref{v_n*} are quite similar. However, a significant difference is that statement (b) requires a uniform exponential growth bound in backward time. This is made possible by the technical condition \ref{delay_const_near_0}, which allows the application of results from \cite{garab:20} for constant delays. 
\medskip

    (a) We prove this part by contradiction. Assume that there exist a sequence $(\varphi_k)_{k=1}^\infty\subset\A\setminus\{\0\}$ and bounded entire solutions $x^k$ of \eqref{eq:sdx} through $\varphi_k$ such that
    \begin{equation}
        \sup_{t\le0}\|x_t^k\|\to0\quad\text{as }k\to\infty,\label{sup_x}
    \end{equation} and $V(x^k_{t_k})>N^*$ for all $k\in\N$, for some $t_k\in\R$, where $x^k=x^{\varphi_k}$ is a vector $x^k(t)=(x^{k,0}(t),\ldots,x^{k,N}(t))$.

    Let $s_k\le 0$ be such that $\|x^k_{s_k}\|>\frac{1}{2}\sup_{t\le 0}\|x^k_t\|$. Using \cite[Theorem 3.3]{balazs:garab:25} and the fact that system \eqref{eq:sdx} is autonomous, we may suppose that $\|x^k_0\|>\frac{1}{2}\sup_{t\le 0}\|x^k_t\|$ and $V(x^k_t)>N^*$ for all $t\le 0$.

    By \cite[Section~2.3]{balazs:garab:25}, $x^k$ solves the nonautonomous linear system
    \begin{alignat*}{3}
        \dot x^{k,i}(t)&=a^k_i(t)x^{k,i}(t)+b^k_i(t)x^{k,i+1}(t),&\qquad& i\in\{0,\dots, N-1\},\\
        \dot x^{k,N}(t)&=a^k_N(t)x^{k,N}(t)+b^k_N(t)x^{k,0}(t-\tau(x^k_t)),&&
    \end{alignat*}
    where $a^k_i$ and $b^k_i$ are defined by
    \begin{equation}\label{def-a-b}
        a^k_i(t)=\int_0^1D_1f^i(hx^{k,i}(t),x^{k,i+1}(t))\intd h,\qquad b^k_i(t)=\int_0^1D_2f^i(0,hx^{k,i+1}(t))\intd h
    \end{equation}
    with $x^{k,N+1}(t)\coloneq x^{k,0}(t-\tau(x^k_t))$. It follows that $a^k$ and $b^k$ are uniformly bounded, and $b^k$ is uniformly bounded away from zero. Since $\|x^k_t\|\to0$ for all $t\le0$ as $k\to\infty$, $f^i$ is continuously differentiable for all $i$, and $\tau\colon \X\to\R$ defined by \eqref{eq:thr_delay} is continuous, we obtain
    \begin{align}
    \lim_{k\to\infty}a^k_i(t)=\mu^i=D_1f^i(0,0),\quad\lim_{k\to\infty}b^k_i(t)=\gamma^i=D_2f^i(0,0),\quad\text{and}\quad\lim_{k\to\infty}\tau(x^k_t)=\tau(\0)\label{lim_a_b_tau}
    \end{align}
    for all $t\le0$ and $i\in\{0,\dots,N\}$.

    Let
    \begin{equation}\label{z-def}
    z^k(t)\coloneq\frac{x^k(t)}{\|x^k_0\|},\quad t\in\R.
    \end{equation}
    Then
    \begin{equation}
        \begin{aligned}
            \dot z^{k,i}(t)&=a^k_i(t)z^{k,i}(t)+b^k_i(t)z^{k,i+1}(t),\\
            \dot z^{k,N}(t)&=a^k_N(t)z^{k,N}(t)+b^k_N(t)z^{k,0}(t-\tau(x^k_t))
        \end{aligned}\label{eq:z}
    \end{equation}
    also holds, and we know that
    $$V(z^k_t)>N^*$$
    and 
    \begin{equation}\label{zk<2}
        \|z^k_t\|<2 \quad \text{for all } t\le0. 
    \end{equation}
    Moreover, $\|z^k_0\|=1$. Hence, the sequence $(z^k)_{k=0}^\infty$ is uniformly bounded, and system \eqref{eq:z}, together with the uniform boundedness of $a^k_i$ and $b^k_i$, implies that  $(z^k)_{k=0}^\infty$ is also equicontinuous. Therefore, by the (classical) Arzel\`a--Ascoli theorem and Cantor's diagonalization argument, we obtain a subsequence $(z^{k_l})_{l=0}^\infty$ and a continuous function $z\colon(-\infty,0]\to\R$ such that
    \begin{align}
        z^{k_l}(t)\to z(t),\quad\text{and}\quad k_l\to\infty\quad\text{as }l\to\infty\label{lim_z_k_l}
    \end{align}
    for all $t\le 0$, and the convergence in \eqref{lim_a_b_tau} is uniform on any compact subset of $(-\infty,0]$. This, together with system \eqref{eq:z} and \eqref{lim_z_k_l}, implies that $z$ is differentiable, $\dot z^{k_l}\to\dot z$ as $l\to\infty$, and the convergence is uniform on compact subsets of $(-\infty,0]$. Furthermore, $z$ satisfies system \eqref{linearized_eq} on the interval $(-\infty,0]$.

    From \eqref{z-def} and \eqref{zk<2} we obtain $\|z_t\|\le2$ for all $t\le 0$, and $\|z_0\|=1$, so $z$ is a nontrivial solution of the constant delay system \eqref{linearized_eq}, that is bounded on $(-\infty,0]$. Thus, \cite[Proposition~3.8]{garab:20} implies $V^{\pm}(z_t,-\tau(\0))\le N^*$ ($V^+$ or $V^-$ in case $\delta=+ 1$ or $\delta=-1$, respectively) for all $t\le 0$.

    Note that thanks to \ref{v_cont} (with constant delay, see \cite[Theorem~3.4]{balazs:garab:25}), there exists $T<0$, such that $z_T\in\cR_{-\tau(\0)}$. 

    Finally, using the local uniform $C^1$-convergence of $(z^{k_l})_{l=0}^\infty$ and \ref{v_cont}, we conclude that
    \begin{equation*}
        N^*<\lim_{l\to\infty}V(x^{k_l}_T)=\lim_{l\to\infty}V^{\pm}(z^{k_l}_T,-\tau(\0))=V^{\pm}(z_T,-\tau(\0))\le N^*,
    \end{equation*}
    which is a contradiction.
    \medskip

    (b) As already mentioned, we need a uniform growth bound to guarantee the boundedness of  the auxiliary variables $z^k$ defined by \eqref{z-def}. We achieve this by applying \cite[Proposition~2.8]{garab:20} multiple times, and \cite[Lemma~2.9]{garab:20}). This requires that the solutions we deal with are defined on an extended interval. For this purpose, let $m$ be fixed such that $m\tau(\0)\geq 2r $ and $\varrho\coloneq -(N^\ast + 4 + m)\tau(\0)$.
    
    We proceed again by contradiction. Assume that there exists a sequence $(\varphi_k)_{k=1}^\infty\subset\A\setminus\{\0\}$ and an entire solution $x^k$ through $\varphi_k$ such that
    \begin{equation}
        \sup_{t\in[\varrho,\infty)}\|x^k_t\|\to 0\quad\text{as }k\to\infty,\label{sup_x}
    \end{equation}
    and $V(x^k_{t_k})<N^*$ for some $t_k\in\R$. Using \ref{delay_const_near_0}, we may assume that 
    \begin{equation}
        \sup_{t\in[\varrho,\infty)}\|x^k_t\|<\eps \label{sup_x_eps}
    \end{equation}
    for all $k\in \N$. Similarly to part (a), we may also assume that 
    \begin{equation}\label{z<2}
        2\|x^k_0\| > \sup_{t\geq 0}\|x_t^k\|
    \end{equation}
    and $V(x^k_{\varrho})<N^\ast$ for all $k$. Then the monotonicity of $V$ also implies $V(x^k_t)<N^\ast$ for all $t\geq \varrho$.
    
    Just as in the first part of the proof, we now obtain that $x^k$ solves the DDE 
    \begin{alignat*}{3}
        \dot x^{k,i}(t)&=a^k_i(t)x^{k,i}(t)+b^k_i(t)x^{k,i+1}(t),&\qquad& i\in\{0,\dots, N-1\},\\
        \dot x^{k,N}(t)&=a^k_N(t)x^{k,N}(t)+b^k_N(t)x^{k,0}(t-\tau(\0)),&&
    \end{alignat*}
    with constant delay for all $t\ge \varrho$, where $a^k_i$ and $b^k_i$ are defined by \eqref{def-a-b}, $a^k_i$ and $b^k_i$ are uniformly bounded, and $b^k_i$ are uniformly bounded away from $0$. Applying \cite[Proposition~2.8]{garab:20} $m$~times (after rescaling time by a factor $\tau(\0)$), and  \cite[Lemma~2.8]{garab:20}, we can derive that there exists $C \geq 1$ such that
    \begin{equation}
        \|x^k_{t}\|\le C^{m+1}\|x^k_{t_0}\|\quad\text{for } t\in [t_0-r , t_0]\label{xk_norm_t0}
    \end{equation}
    for all $t_0\geq 0$.
    
    As in part (a), the function $z^k$, defined by \eqref{z-def},
    satisfies the linear constant delay system
    \begin{alignat}{3}\label{z-bontas}
        \dot z^{k,i}(t)&=a^k_i(t)z^{k,i}(t)+b^k_i(t)z^{k,i+1}(t),&\qquad& i\in\{0,\dots ,N-1\},\\
        \dot z^{k,N}(t)&=a^k_Nz^{k,N}(t)+b^k_{k,N}z^{k,0}(t-\tau(\0)).&&
    \end{alignat}
    for all $k\in\N$ and $t\geq \varrho$.

   From \eqref{z<2} and \eqref{xk_norm_t0}, we obtain 
   $$\|z^k_t\|\leq \max\{2,C^{m+1}\}$$
   for all $k\in \N$ and $t\geq t-r $, so $\dot{z}^k$ is also uniformly bounded on $[0,\infty)$. This implies that functions $z^k$ are also equicontinuous, and we can deduce just as in part (a) that there exist locally uniformly convergent subsequences
    \[z^{k_l} \to z,\quad \mbox{and}\quad \dot{z}^{k_l} \to \dot{z}.\]
    It follows that $z$ satisfies the linear equation \eqref{linearized_eq} on $[0,\infty)$. Since $\|z^{k_l}_0\|=1$ for all $l\in \N$, we have $\|z_0\| = 1$. Therefore, $z$ is a nontrivial solution of equation \eqref{linearized_eq} that is bounded on $[0,\infty)$, and we can apply \cite[Lemma 2.7]{garab:20} to obtain that $V^\pm(z_t,-\tau(\0)]) \geq N^\ast$ for all $t\in [0,\infty)$, and by the monotonicity of $V^\pm$, this also holds for all $t \in \R $.
    
    Finally, we obtain the existence of some $T> 0$, such that $z_T \in \cR_{-\tau(\0)}$. Then, using the local uniform $C^1$-convergence of $(z^{k_l})_{l=0}^\infty$, we arrive at 
    \begin{equation*}
        N^\ast > \lim_{l \to \infty} V(x^{k_l}_T) = \lim_{l \to \infty} V^{\pm}(z^{k_l}_T,-\tau(\0))= V^{\pm}(z_T,-\tau(\0)) \geq N^\ast,
    \end{equation*}
    which is a contradiction.
\end{proof}

\subsection{Relation to other results and possible generalizations}\label{subsec:covers-prev-results}
Let us briefly demonstrate that several similar results on delay differential and difference equations (i.e. \cite{polner:02, mallet-paret:88-morse, garab:poetzsche:19, garab:20,bartha:garab:krisztin:25}) on the existence of Morse decompositions can be obtained as special cases of \cref{th:morse}. 

\subsubsection*{Application to delay difference equations}

Let us consider the difference equation
\[x_{k+1}=f(x_k,x_{k-n}),\]
where $f\in C^1(\R^2,\R)$ fulfills some feedback conditions as in \cite[(${\bf H}_{\bf 1})$--(${\bf H}_{\bf 3})$]{garab:poetzsche:19}, and assume that it has a global attractor -- conditions guaranteeing this can be found in \cite[Section 5]{garab:poetzsche:19} or in \cite{garab:18}.

Let $\T\coloneq \Z$, $\widetilde{\X}\coloneq \X\coloneq \R^{n+1}$, $\O \coloneq  \{(0,\dots,0)\}$, and $d=\widetilde{d}$ be the Euclidean distance. Moreover, let the discrete Lyapunov function $V$ (counting the number of sign changes in ($n+1$)-dimensional vectors) and $\cR$ be defined as in \cite[pp.\ 909--910]{garab:poetzsche:19}. Note that  \ref{cond-o-exists}--\ref{att_d} are trivially satisfied. 

Furthermore, \cite[Proposition 3.2]{garab:poetzsche:19} yields conditions \ref{v_usc} and \ref{v_cont}, \cite[Theorem 3.3]{garab:poetzsche:19} guarantees conditions \ref{v_dec} and \ref{v_r} (with $T=4n+2$), while Proposition 4.3 of the same paper ensures that \ref{v_n*} holds (where $N^\ast$ is defined on p.\ 914, and is closely related to the number of unstable eigenvalues of the linearized equation).

Thus, \cref{th:morse} applies and yields a Morse decomposition of the global attractor. Moreover, choosing $N_0$ large enough (namely, $N_0>n+1$) trivially yields that $\S_{N_0}^+$ is the empty set, and in that case the Morse decomposition given by \cref{th:morse} and by \cite[Theorem 4.1]{garab:poetzsche:19} coincide.

\subsubsection*{Application to cyclic DDEs with constant delay}
System \eqref{eq:sdx} with a constant delay was studied in \cite[Section~3]{garab:20}, where a Morse decomposition was also constructed (see Theorem~3.2 therein), generalizing the corresponding results of Mallet-Paret \cite{mallet-paret:88-morse} and Polner \cite{polner:02} for scalar equations. Although condition \ref{delay_const_near_0} is restrictive, it certainly allows constant delays, so \cref{thm:sdde} can be applied. 

It is worth noting that in \cite{garab:20}, boundedness of the corresponding Lyapunov function was also proved, and a Morse set like $\S_{N_0}^+$ was unnecessary in the construction. Now, if we choose $N_0$ larger than this upper bound, then \cref{thm:sdde} gives exactly the same Morse decomposition as in \cite{garab:20}.

\subsubsection*{Scalar SDDEs with more general delays}

The proof of \cref{thm:sdde} uses condition \ref{delay_const_near_0} only to verify \ref{v_n*}. The latter condition is shown in a recent paper \cite[Lemma~4.4]{bartha:garab:krisztin:25} for the scalar case with negative feedback, assuming only boundedness and Lipschitz continuity of $\tau$, and strict monotonicity of the delayed argument. This means that -- using the same notation as in \cref{subsec:sdde} -- its main theorem (Theorem~4.2) is also a special case of \cref{th:morse}.

\subsubsection*{Possible generalizations}
We conjecture that \cref{thm:sdde} can be extended to  cyclic SDDE systems under more general assumptions on the delay, similarly to the scalar case \cite[Theorem~4.2]{bartha:garab:krisztin:25}. As demonstrated above, the only missing link is to prove condition \ref{v_n*}. This boils down to showing that for any solution $x$ with $V(x_t)$ less than a fixed value, there exists a uniform exponential growth bound backward in time (cf.\ inequality \eqref{xk_norm_t0}).  This would also imply that nontrivial solutions with finitely many sign changes on any bounded interval are not superexponential (i.e.\ they do not tend to zero faster than exponentially). Combining this with the results of \cite{balazs:garab:25} (Theorem~3.9 and Section~4.1), one obtains that there are no superexponential solutions on the global attractor. Since the mentioned results of \cite{balazs:garab:25} apply for a wider class of state-dependent delays, it would be desirable to show such growth bounds in this general setting. This seems nontrivial; however, one promising way forward may be to combine the proofs of \cite[Proposition 2.8]{garab:20}, carried out for cyclic DDEs with constant delay, and \cite[Lemmas~3.3--3.4]{bartha:garab:krisztin:25}, which deals with scalar SDDEs. This can serve as a topic of future work.

\section*{Acknowledgment}
This research was supported by the National Research, Development and Innovation (NRDI) Fund, Hungary, [project no.\ TKP2021-NVA-09 and FK~142891], and by the National Laboratory for Health Security [RRF-2.3.1-21-2022-00006]. I.~B. was supported by NRDI no.~FK~138924. Á.~G.\ was supported by the János Bolyai Research Scholarship of the Hungarian Academy of Sciences.

\section*{Declarations}

\begin{itemize}
\item Conflict of interest: The authors declare no competing interests.
\item Ethical Approval: Not applicable.
\item Consent to Participate: Not applicable.
\item Consent to Publish: Not applicable.
\end{itemize}

\printbibliography

@article{avraham:Sharon:zarai:margaliot:20,
 author = {Avraham, Tsuff Ben and Sharon, Guy and Zarai, Yoram and Margaliot, Michael},
 title = {Dynamical systems with a cyclic sign variation diminishing property},
 fjournal = {IEEE Transactions on Automatic Control},
 journal = {IEEE Trans. Autom. Control},
 issn = {0018-9286},
 volume = {65},
 number = {3},
 pages = {941--954},
 year = {2020},
 doi = {10.1109/TAC.2019.2914976},
 url = {https://doi.org/10.1109/TAC.2019.2914976},
 zbMATH = {7256223},
 Zbl = {1533.93285}
}

@unpublished{balazs:garab:25,
    AUTHOR = {Bal\'azs, Istv\'an and Garab, \'Abel},
     TITLE = {Discrete Lyapunov functional for cyclic systems of differential equations with time-variable or state-dependent delay},
    note = {arXiv:2502.03648}
}

@article{bartha:garab:krisztin:25,
      title={Morse decomposition of scalar differential equations with state-dependent delay}, 
      author={Ferenc A. Bartha and Ábel Garab and Tibor Krisztin},
      journal={J. Dyn. Differential Equations},
      note={Published online},
      year={2025},
      doi={10.1007/s10884-025-10414-w},
      url={https://doi.org/10.1007/s10884-025-10414-w}, 
}

@article {brunovsky:fiedler:86,
    AUTHOR = {Brunovsk\'y, Pavol and Fiedler, Bernold},
     TITLE = {Numbers of zeros on invariant manifolds in reaction-diffusion
              equations},
   JOURNAL = {Nonlinear Anal.},
  FJOURNAL = {Nonlinear Analysis. Theory, Methods \& Applications. An
              International Multidisciplinary Journal},
    VOLUME = {10},
      YEAR = {1986},
    NUMBER = {2},
     PAGES = {179--193},
      ISSN = {0362-546X,1873-5215},
   MRCLASS = {35K57 (35B05 35B40)},
  MRNUMBER = {825216},
MRREVIEWER = {Hiroshi\ Matano},
       DOI = {10.1016/0362-546X(86)90045-3},
       URL = {https://doi.org/10.1016/0362-546X(86)90045-3},
}

@book{colonius:kliemann:00,
 author = {Colonius, Fritz and Kliemann, Wolfgang},
 title = {The dynamics of control. {With} an appendix by {Lars} {Gr{\"u}ne}},
 fseries = {Systems \& Control: Foundations \& Applications},
 series = {Syst. Control Found. Appl.},
 issn = {2324-9749},
 isbn = {0-8176-3683-8},
 year = {2000},
 publisher = {Boston: Birkh{\"a}user},
 language = {English},
 zbMATH = {1526959},
 Zbl = {1020.93500},
 doi = {10.1007/978-1-4612-1350-5},
 url = {https://doi.org/10.1007/978-1-4612-1350-5}
}

@book{conley:78,
 author = {Conley, Charles},
 title = {Isolated invariant sets and the {Morse} index},
 fseries = {Regional Conference Series in Mathematics},
 series = {Reg. Conf. Ser. Math.},
 issn = {0160-7642},
 volume = {38},
 year = {1978},
 publisher = {American Mathematical Society (AMS), Providence, RI},
 language = {English},
 zbMATH = {3616736},
 Zbl = {0397.34056}
}

@article {garab:18,
    AUTHOR = {Garab, \'Abel},
     TITLE = {A note on dissipativity and permanence of delay difference
              equations},
   JOURNAL = {Electron. J. Qual. Theory Differ. Equ.},
  FJOURNAL = {Electronic Journal of Qualitative Theory of Differential
              Equations},
      YEAR = {2018},
     PAGES = {Paper No. 51, pp. 1--12},
      ISSN = {1417-3875},
   MRCLASS = {39A22 (92D25)},
  MRNUMBER = {3827989},
       DOI = {10.14232/ejqtde.2018.1.51},
       URL = {https://doi.org/10.14232/ejqtde.2018.1.51},
}

@article {garab:20,
    AUTHOR = {Garab, \'Abel},
     TITLE = {Absence of small solutions and existence of {M}orse
              decomposition for a cyclic system of delay differential
              equations},
   JOURNAL = {J. Differential Equations},
  FJOURNAL = {Journal of Differential Equations},
    VOLUME = {269},
      YEAR = {2020},
    NUMBER = {6},
     PAGES = {5463--5490},
      ISSN = {0022-0396,1090-2732},
   MRCLASS = {34K12 (37B25 37B35 37C70)},
  MRNUMBER = {4104475},
MRREVIEWER = {I.\ P.\ Stavroulakis},
       DOI = {10.1016/j.jde.2020.04.014},
       URL = {https://doi.org/10.1016/j.jde.2020.04.014},
}

@article {garab:poetzsche:19,
	AUTHOR = {Garab, {\'A}bel and P\"{o}tzsche, Christian},
	TITLE = {Morse decompositions for delay-difference equations},
	JOURNAL = {J. Dynam. Differential Equations},
	FJOURNAL = {Journal of Dynamics and Differential Equations},
	VOLUME = {31},
	YEAR = {2019},
	NUMBER = {2},
	PAGES = {903--932},
	ISSN = {1040-7294},
	MRCLASS = {37B25 (34K28 37C70 39A30)},
	MRNUMBER = {3951829},
	DOI = {10.1007/s10884-018-9685-8},
	URL = {https://doi.org/10.1007/s10884-018-9685-8},
}

@book{hale:88,
  title={Asymptotic behavior of dissipative systems},
  author={Hale, Jack K.},
  number={25},
  year={1988},
  publisher={American Mathematical Soc.},
  doi={10.1090/surv/025},
  url={https://doi.org/10.1090/surv/025}
}

@book {kloeden:rasmussen:11,
    AUTHOR = {Kloeden, Peter E. and Rasmussen, Martin},
     TITLE = {Nonautonomous dynamical systems},
    SERIES = {Mathematical Surveys and Monographs},
    VOLUME = {176},
 PUBLISHER = {American Mathematical Society, Providence, RI},
      YEAR = {2011},
     PAGES = {viii+264},
      ISBN = {978-0-8218-6871-3},
   MRCLASS = {37B55 (37C60 37C75 37D10 37H05)},
  MRNUMBER = {2808288},
MRREVIEWER = {Rafael\ Obaya Garc\'ia},
       DOI = {10.1090/surv/176},
       URL = {https://doi.org/10.1090/surv/176}
}

@article {mallet-paret:88-morse,
	AUTHOR = {Mallet-Paret, John},
	TITLE = {Morse decompositions for delay-differential equations},
	JOURNAL = {J. Differential Equations},
	FJOURNAL = {Journal of Differential Equations},
	VOLUME = {72},
	YEAR = {1988},
	NUMBER = {2},
	PAGES = {270--315},
	ISSN = {0022-0396},
	MRCLASS = {58F32 (34K05 58F09 58F25)},
	MRNUMBER = {932368},
	MRREVIEWER = {S. G. Deo},
	DOI = {10.1016/0022-0396(88)90157-X},
	URL = {https://doi.org/10.1016/0022-0396(88)90157-X},
}

@article {fiedler:mallet-paret:89,
    AUTHOR = {Fiedler, Bernold and Mallet-Paret, John},
     TITLE = {A {P}oincar\'e--{B}endixson theorem for scalar reaction
              diffusion equations},
   JOURNAL = {Arch. Rational Mech. Anal.},
  FJOURNAL = {Archive for Rational Mechanics and Analysis},
    VOLUME = {107},
      YEAR = {1989},
    NUMBER = {4},
     PAGES = {325--345},
      ISSN = {0003-9527},
   MRCLASS = {35K57 (35B40 58F05 58F12)},
  MRNUMBER = {1004714},
MRREVIEWER = {Gunter\ Bengel},
       DOI = {10.1007/BF00251553},
       URL = {https://doi.org/10.1007/BF00251553},
}

@article {mallet-paret:sell:96-lyapunov,
	AUTHOR = {Mallet-Paret, John and Sell, George R.},
	TITLE = {Systems of differential delay equations: {F}loquet multipliers
	and discrete {L}yapunov functions},
	JOURNAL = {J. Differential Equations},
	FJOURNAL = {Journal of Differential Equations},
	VOLUME = {125},
	YEAR = {1996},
	NUMBER = {2},
	PAGES = {385--440},
	ISSN = {0022-0396},
	CODEN = {JDEQAK},
	MRCLASS = {34K15 (34C35 58F32)},
	MRNUMBER = {1378762 (97a:34193a)},
	MRREVIEWER = {Hal Leslie Smith},
	DOI = {10.1006/jdeq.1996.0036},
	URL = {http://dx.doi.org/10.1006/jdeq.1996.0036},
}

@incollection {mallet-paret:sell:03,
    AUTHOR = {Mallet-Paret, John and Sell, George R.},
     TITLE = {Differential systems with feedback: time discretizations and
              {L}yapunov functions},
      NOTE = {Special issue dedicated to Victor A. Pliss on the occasion of
              his 70th birthday},
   JOURNAL = {J. Dynam. Differential Equations},
  FJOURNAL = {Journal of Dynamics and Differential Equations},
    VOLUME = {15},
      YEAR = {2003},
    NUMBER = {2-3},
     PAGES = {659--698},
      ISSN = {1040-7294,1572-9222},
   MRCLASS = {37C65 (34K07 34K28 35R10)},
  MRNUMBER = {2046733},
MRREVIEWER = {Tom\'a\v s\ Gedeon},
       DOI = {10.1023/B:JODY.0000009750.14308.09},
       URL = {https://doi.org/10.1023/B:JODY.0000009750.14308.09},
}

@book {munkres:14,
    AUTHOR = {Munkres, James R.},
     TITLE = {Topology: a first course},
 PUBLISHER = {Prentice-Hall, Inc., Englewood Cliffs, NJ},
      YEAR = {1975},
     PAGES = {xvi+413},
   MRCLASS = {54-01},
  MRNUMBER = {464128},
}

@article {polner:02,
	AUTHOR = {Polner, M\'{o}nika},
	TITLE = {Morse decomposition for delay-differential equations with
	positive feedback},
	JOURNAL = {Nonlinear Anal.},
	FJOURNAL = {Nonlinear Analysis. Theory, Methods \& Applications. An
	International Multidisciplinary Journal},
	VOLUME = {48},
	YEAR = {2002},
	NUMBER = {3, Ser. A: Theory Methods},
	PAGES = {377--397},
	ISSN = {0362-546X},
	MRCLASS = {37B35 (34K99 37B25 37C70)},
	MRNUMBER = {1869518},
	MRREVIEWER = {Tom\'{a}\v{s} Gedeon},
	DOI = {10.1016/S0362-546X(00)00191-7},
	URL = {https://doi.org/10.1016/S0362-546X(00)00191-7},
}

@phdthesis{terevsvcak:1994,
  title={Dynamical systems with discrete Lyapunov functionals},
  author={Tere{\v{s}}{\v{c}}{\'a}k, Ign{\'a}c},
  year={1994},
  school={Comenius University, Bratislava}
}

@article{van-den-berg:munao:vandervorst:16,
 author = {van den Berg, J. B. and Muna{\`o}, S. and Vandervorst, R. C. A. M.},
 title = {The {Poincar{\'e}}--{Bendixson} theorem and the non-linear {Cauchy}--{Riemann} equations},
 fjournal = {Physica D},
 journal = {Physica D},
 issn = {0167-2789},
 volume = {334},
 pages = {19--28},
 year = {2016},
 doi = {10.1016/j.physd.2016.04.009},
 url = {https://doi.org/10.1016/j.physd.2016.04.009},
 zbMATH = {7074445},
 Zbl = {1415.35051}
}

@article {yan:zhou:25,
    AUTHOR = {Yan, Kaige and Zhou, Dun},
     TITLE = {Constancy of discrete {L}yapunov functionals on
              stable/linearly stable minimal sets and its applications},
   JOURNAL = {Proc. Amer. Math. Soc.},
  FJOURNAL = {Proceedings of the American Mathematical Society},
    VOLUME = {153},
      YEAR = {2025},
    NUMBER = {6},
     PAGES = {2419--2432},
      ISSN = {0002-9939,1088-6826},
   MRCLASS = {34C46 (37C10 37C60 37C75)},
  MRNUMBER = {4892617},
       DOI = {10.1090/proc/16703},
       URL = {https://doi.org/10.1090/proc/16703},
}

\end{document}